\documentclass[10pt,a4paper]{amsart}

\usepackage{lmodern}
\usepackage[utf8]{inputenc}
\usepackage[T1]{fontenc}

\usepackage{fancyref}
\usepackage[colorlinks=true]{hyperref}

\usepackage{amsmath}
\usepackage{amssymb}
\usepackage{amsthm}

\newtheorem{theo}{Theorem}
\newtheorem{prop}[theo]{Proposition}

\def\ov#1{\overline{#1}}

\author{Gunnar Þór Magnússon}
\address{Hafnarfjörður, Iceland}
\email{gunnar@magnusson.io}
\date{\today}
\title{A remark on Wu's remark on\\holomorphic sectional curvature}

\begin{document}

\maketitle

Wu \cite{wu1973remark} proved that if $g$ and $h$ are Hermitian metrics on a
complex manifold and their holomorphic sectional curvatures satisfy $H_g \leq
-K_g < 0$ and $H_h \leq -K_h < $ then the holomorphic sectional curvature of
their sum satisfies
\[
H_{g + h} \leq \frac{-K_g K_h}{K_g + K_h} < 0.
\]
The proof shows
that given any tangent vector $\xi \in T_{X,x}$ there
exists a holomorphic embedding $f : D \to X$ of the unit disk in $X$ such that
$f(0) = x$ and $f'(0) = \xi$ and such that the holomorphic sectional curvature
of a given metric at $\xi$ is the curvature of the pullback metric
to the disk at its center.
Thus we reduce to the case of one complex variable,
which had already been proven by Grauert and
Reckziegel~\cite{grauert1965hermitesche}.

In this note we point out that there is a direct route to Wu's theorem by using
an expression for the curvature tensor of the sum of Hermitian metrics one
obtains from Griffiths' theorem on the curvatures of sub- and quotient
bundles~\cite{griffiths1965hermitian}. Our starting point is the following
result, which is given as an exercise in Zheng's lovely
textbook~\cite{zheng2000complex}.

\begin{prop}
Let $E \to X$ be a holomorphic vector bundle over a complex manifold.
Let $g$ and $h$ be Hermitian metrics on $E$.
Then the curvature tensor of $g + h$~is%
\[
R_{g + h} = R_g + R_h - \sigma^*q,
\]
where $q$ is a Hermitian metric on $E$ and $\sigma(\xi)s = D_{g,\xi} s -
D_{h,\xi} s$.
\end{prop}

\begin{proof}
Consider short exact sequence
\[
0 \longrightarrow
E \longrightarrow
E \oplus E \longrightarrow
E \longrightarrow
0,
\]
where the first nonzero arrow is $j(s) = s \oplus s$ and the second is $\pi(s
\oplus t) = s - t$.
We equip $E \oplus E$ with the metric $g \oplus h$, the subbundle with its
pullback $g + h$, and the quotient with the induced metric $q$.
The second fundamental form of $E$ in $E \oplus E$ is then
\[
\sigma(\xi) s
= \pi(D_{g \oplus h, \xi} j(s))
= D_{g,\xi}s - D_{h,\xi} s.
\]
The curvature tensor of $g + h$ is then $j^*R_{g \oplus h} - \sigma^* q$, which
unravels to what we wanted to show.
\end{proof}

It is possible to write down exactly what this ``quotient'' metric $q$ is; in fact
\[
q = ((g+h)^{-1}h)^*g
+ ((g+h)^{-1}g)^*h,
\]
where we view Hermitian metrics on $T_X$ as smooth isomorphisms $T_X \to
\ov{T}_X$. We won't need to know this, but it is useful for analyzing what
happens when one of the metrics varies and leads to the basic estimate $q \leq g
+ h$.

\begin{theo}[Wu~\cite{wu1973remark}]
Let $g$ and $h$ be Hermitian metrics on a complex manifold $X$ and assume both
have nonpositive holomorphic sectional curvature.
Then the holomorphic sectional curvature of $g + h$ is nonpositive, and satisfies
\[
H_{g + h} \leq \frac{H_g H_h}{H_g + H_h} < 0
\]
if the holomorphic sectional curvatures of $g$ and $h$ are negative.
\end{theo}

\begin{proof}
We know that
\[
R_{g + h} = R_g + R_h - \sigma^* q
\]
and that $q$ is positive-definite.
Then
\[
H_{g+h}(\xi)
\leq \frac{H_g(\xi) |\xi|^4_g + H_h(\xi) |\xi|^4_h}{|\xi|^4_{g + h}}
= \frac{H_g(\xi) |\xi|^4_g + H_h(\xi) |\xi|^4_h}{(|\xi|^2_{g} + |\xi|^2_h)^2}.
\]
Thus $H_{g + h} \leq 0$ if $H_g \leq 0$ and $H_h \leq 0$.

As Wu notes,
for any positive real numbers $a, b, x, y$ we have
\[
\frac{xy}{x + y} \leq \frac{x a^2 + y b^2}{(a + b)^2}
\]
by elementary algebra.
The inequality is reversed if $x$ and $y$ are negative.
If the holomorphic sectional curvatures of $g$ and $h$ are negative we thus get
\[
H_{g + h} \leq \frac{H_g H_h}{H_g + H_h} < 0.
\qedhere
\]
\end{proof}

One can show that if a metric on a compact manifold has negative (or positive)
holomorphic sectional curvature, then so does any sufficiently small
deformation of the metric.
On a K\"ahler manifold this implies that the set of K\"ahler classes with a
representative with negative holomorphic sectional curvature is an open cone
inside the K\"ahler cone.
Wu's theorem further shows that convex combinations of metrics in this cone are
again in it, so the cone is connected.
It would be unexpected if these two cones were the same but I don't know of an
example where they differ.

\bibliographystyle{plain}
\bibliography{main}

\end{document}